\documentclass[a4paper,11pt,french,english]{amsart}
\usepackage[english]{babel}
\usepackage[utf8]{inputenc}
\usepackage{lmodern}
\usepackage[T1]{fontenc}
\usepackage{amsmath,amsthm,amssymb,amsfonts}
\usepackage{amscd}
\usepackage{pstricks}
\usepackage{pst-node}
\usepackage[a4paper,left=3.4cm,right=3.4cm,top=4cm,bottom=4cm]{geometry}
\usepackage{enumerate}
\usepackage{enumitem}
\usepackage[linktocpage=true]{hyperref}
\usepackage{url}
\usepackage{csquotes}
\usepackage{tikz-cd}

\title[]{Free abelian quotients of commensurators}
\author{Adrien Le Boudec}
\date{September 28, 2026.}
\address{CNRS, UMPA - ENS Lyon, 46 all\'ee d'Italie, 69364 Lyon, France}
\email{adrien.le-boudec@ens-lyon.fr}

\newcommand{\Z}{\mathbb{Z}}
\newcommand{\Q}{\mathbb{Q}}
\newcommand{\R}{\mathbb{R}}
\newcommand{\A}{\mathbb{A}}

\newcommand{\Aut}{\mathrm{Aut}}
\newcommand{\Comm}{\mathrm{Comm}}
\newcommand{\AComm}{\mathrm{AComm}}

\newcommand{\GL}{\mathrm{GL}}
\newcommand{\SL}{\mathrm{SL}}
\newcommand{\PGL}{\mathrm{PGL}}
\newcommand{\PSL}{\mathrm{PSL}}

\newcommand{\simple}{\mathcal{FS}}
\newcommand{\simpleab}{\mathcal{FS}_{ab}}

\theoremstyle{plain}
\newtheorem{Theorem}{Theorem}[section]
\newtheorem{Proposition}[Theorem]{Proposition}
\newtheorem{Corollary}[Theorem]{Corollary}
\newtheorem{Lemma}[Theorem]{Lemma}
\newtheorem{Step}{Step}

\newtheorem{Theorem-intro}{Theorem}
\newtheorem{Corollary-intro}[Theorem-intro]{Corollary}

\theoremstyle{definition}
\newtheorem{Definition}[Theorem]{Definition}
\newtheorem{Notation}[Theorem]{Notation}
\newtheorem{Question}[Theorem]{Question}

\newtheorem{Remark}[Theorem]{Remark}

\begin{document}

\maketitle


\begin{abstract}
For every group $\Gamma$, we define a homomorphism $d^\Gamma : \Comm(\Gamma) \to \Z^{(\simple)}$ from the abstract commensurator $\Comm(\Gamma)$ to the free abelian group $\Z^{(\simple)}$ with basis the collection $\simple$ of isomorphism classes of finite simple groups. We investigate the homomorphism $d^\Gamma$ when $\Gamma$ is a finitely generated free group $F$. We explicitly describe the image of $d^F : \Comm(F) \to \Z^{(\simple)}$, which is a free abelian group of infinite rank, and we show that the kernel is the monolith of the group $\Comm(F)$. We deduce in particular that every proper quotient of $\Comm(F)$ is abelian. 

We use this to study the commensurator of a cocompact lattice in the automorphism group of a regular tree, which can be seen as a subgroup of $\Comm(F)$. We show that the image of this group under $d^F$ is again a free abelian group of infinite rank, showing in particular this group is not virtually simple. 
\end{abstract}


\section{Introduction}

\subsection{The commensurator of a cocompact tree lattice}

Let $T_d$ be a regular tree of degree $d \geq 3$. The group $\Aut(T_d)$ of automorphisms of $T_d$ is a locally compact and totally disconnected group. By a theorem of Tits, the index two subgroup of $\Aut(T_d)$ consisting of type preserving automorphisms of $T_d$ is a simple group \cite{Tits_arbre}. Recall that for a locally compact group $G$ and a discrete subgroup $\Gamma$, the commensurator of $\Gamma$ in $G$ is the subgroup $\Comm_G(\Gamma)$ consisting of elements $g \in G$ such that $\Gamma$ and $g \Gamma g^{-1}$ are commensurable; where two subgroups are called commensurable if their intersection has finite index in each one of them. The study of lattices and their commensurators in automorphism group of trees emerged at the beginning of the 90's, partly motivated by  developments about  lattices and their commensurators in connected semi-simple Lie groups. Two classical results due to Margulis in the latter situation are the arithmeticity criterion and the commensurator super-rigidity theorem. The  arithmeticity criterion asserts that if $\Gamma$ is an irreducible lattice in a connected semi-simple Lie group $G$ with trivial center and no compact factor, then $\Gamma$ is arithmetic if and only if $\Comm_G(\Gamma)$ is a dense subgroup of $G$ \cite[Chap.\ IX (1.9)]{Margulis-book}. The commensurator super-rigidity theorem asserts that every Zariski dense linear representation of $\Comm_G(\Gamma)$ in a simple algebraic group over a local field is either bounded in restriction to $\Gamma$, or extends to a continuous representation of $G$  \cite[Chap.\ VII Th.\ 5.4]{Margulis-book}.

For $G = \Aut(T_d)$, the automorphism group of a regular tree, a theorem of Leighton and Bass--Kulkarni asserts that if $\Lambda, \Gamma$ are cocompact lattices in $G$, then there is $g \in G$ such that $g \Lambda g^{-1}$ and $\Gamma$ are commensurable \cite{Leighton-finite-cover, Bass-Kulk-90}.  In particular $\Comm_G(\Lambda)$ and $\Comm_G(\Gamma)$ are conjugate in $G$. Hence there is, up to isomorphism, only one commensurator of a cocompact lattice in $G$.  Bass--Kulkarni  showed that $\Comm_G(\Gamma)$  is dense in $G$ \cite{Bass-Kulk-90}. Lubotzky--Mozes--Zimmer showed a commensurator super-rigidity theorem, namely that any minimal action of $\Comm_G(\Gamma)$ on another tree $T$ is either bounded in restriction to $\Gamma$, or extends to a continuous action of $G$ \cite{LMZ-superrigidity}. Lubotzky raised the problem whether $\Gamma$ has the congruence subgroup property in $\Comm_G(\Gamma)$, and this was solved in the affirmative by Mozes  \cite{Mozes-CSP-tree}.

Pursuing the similarity with the setting of arithmetic groups and by analogy with the Margulis-Platonov conjecture, Lubotzky--Mozes--Zimmer raised the problem whether the type preserving index two subgroup $\Comm_G(\Gamma)^+$ of $\Comm_G(\Gamma)$  is a simple group  \cite{LMZ-superrigidity}. Recall that a group $C$ is called \textbf{monolithic} if $\mathrm{Mon}(C) := \bigcap_{N \neq 1} N$, the intersection of all  non-trivial normal subgroups of $C$, is non-trivial. When this holds, $\mathrm{Mon}(C)$ is called the \textbf{monolith} of $C$, and it is the unique smallest non-trivial normal subgroup of $C$. Lubotzky--Mozes--Zimmer and Caprace showed the group $\Comm_G(\Gamma)$ is monolithic, and that the monolith of $\Comm_G(\Gamma)$ contains a finite index subgroup of $\Gamma$ \cite[Proposition 5.1]{LMZ-superrigidity}, \cite[Theorem A.1]{Cap-appendix-comm-tree}. Moreover Caprace showed that the monolith of $\Comm_G(\Gamma)$ is a simple group \cite[Theorem A.1]{Cap-appendix-comm-tree}. In that language the problem whether $\Comm_G(\Gamma)^+$ is a simple group is equivalent to asking whether $\Comm_G(\Gamma)^+$ is the monolith of $\Comm_G(\Gamma)$. We answer this problem in the negative. More precisely, we prove the following theorem:

\begin{Theorem-intro} \label{thm-intro-comme-tree-abelian-quotient}
	Let $\Gamma$ be a cocompact lattice in $G = \Aut(T_d)$. Then the commensurator $\Comm_G(\Gamma)$ of $\Gamma$ in $G$ admits a proper quotient that is free abelian of infinite rank. In particular $\Comm_G(\Gamma)$  is not virtually simple.
\end{Theorem-intro}

For $\Gamma$ and $C = \Comm_G(\Gamma)$ as in the theorem, there is a natural locally compact group associated to the pair $(\Gamma,C)$, namely the  completion $ C / \! \! / \Gamma$  of $C$ with respect to $\Gamma$. The group $C / \! \! / \Gamma$ is implicit in \cite{Mozes-CSP-tree}. It appears explicitly in the work of Creutz--Shalom \cite{Creutz-Shalom-NST}, as well as in the aforementioned work of Caprace \cite{Cap-appendix-comm-tree}. More recently, the group $ C / \! \! / \Gamma$ (and more generally the completion of a subgroup containing and commensurating $\Gamma$) played an important role in the work of the author and C.\ Reid \cite{LB-Reid-latt-determined}. Theorem \ref{thm-intro-comme-tree-abelian-quotient} above can be equivalently stated saying that the group $ C / \! \! / \Gamma$  admits an open normal subgroup such that the associated quotient is a free abelian group of infinite rank: see Corollary  \ref{cor-completion-abelian-quotient}.

\subsection{The abstract commensurator of a free group}

In the course of the proof of Theorem  \ref{thm-intro-comme-tree-abelian-quotient}, the group $\Comm_G(\Gamma)$  is profitably seen as being a subgroup of the abstract commensurator of $\Gamma$. We recall the definitions. If $\Gamma$ is a group, an isomorphism $\varphi:A\to B$ between finite index subgroups $A,B$ of $\Gamma$ is called a {virtual isomorphism of $\Gamma$}. The abstract commensurator of $\Gamma$ consists of virtual isomorphisms of $\Gamma$, modulo the equivalence relation that identifies two of them if they agree on a common finite index subgroup of their domains (see \S  \ref{subsec-prelim-comm}). It forms a group, denoted $\Comm(\Gamma)$, whose isomorphism class is insensitive to passing from $\Gamma$ to a finite index subgroup.

If $\Gamma$ is a subgroup of a group $G$, there is a natural homomorphism $ \Comm_G(\Gamma) \to \Comm(\Gamma)$. In certain situations, rigidity results assert that this homomorphism is surjective, or at least that the image is a finite index subgroup of $\Comm(\Gamma)$. An iconic illustration of such a result is Mostow rigidity: whenever  $\Gamma$ is an irreducible lattice in a connected semi-simple Lie group $G$ with finite center and no compact factor, and $G$ is not locally isomorphic to $\mathrm{SL}(2,\R)$, then $ \Comm_G(\Gamma) \to \Comm(\Gamma)$ has finite index image \cite{Margulis-book}. By contrast, for $\Gamma$ and $G$ as in Theorem  \ref{thm-intro-comme-tree-abelian-quotient}, the group $\Gamma$ is virtually a free group, and the abstract commensurator of $\Gamma$ (which is isomorphic to the abstract commensurator of a free group) stands away from any kind of rigidity result of this form. In the sequel we denote by $F$ a finitely generated free group of rank at least two. Again, since two finitely generated free groups of different ranks always admit isomorphic finite index subgroups, the group $\Comm(F)$ does not depend on the rank of $F$.

 \begin{Definition} \label{def-intro-AComm}
If $A$ is a finite index subgroup of a group $\Gamma$, then any automorphism of $A$ defines a virtual isomorphism of $\Gamma$. The associated map $a_A: \Aut(A) \to \Comm(\Gamma)$ is a homomorphism.  	We denote by $\AComm(\Gamma)$ the subgroup of $\Comm(\Gamma)$ generated by the subgroups $a_A(\Aut(A))$, when $A$ ranges over finite index subgroups of $\Gamma$. 
 \end{Definition}

When $\Gamma$ is a finitely generated group, $\AComm(\Gamma)$ is a normal subgroup of  $\Comm(\Gamma)$ \cite[Proposition 1.4.2]{Odden-thesis},  \cite[Lemma 2.22]{BELBRVW-comm-free}. In the recent work of Barnea, Ershov, Reid, Vannacci, Weigel and the author  \cite{BELBRVW-comm-free}, a sufficient condition on $\Gamma$ was established ensuring that the quotient of $\Comm(\Gamma)$ by $\AComm(\Gamma)$ is infinite (see \cite[Corollary 5.6]{BELBRVW-comm-free}). That condition was applied there to the free group case, showing that the quotient of $\Comm(F)$ by $\AComm(F)$ is an infinite group. The following result identifies the structure of the quotient, answering a problem raised in \cite[Question 5]{BELBRVW-comm-free}: 

\begin{Theorem-intro} \label{thm-intro-comm-free}
Let $F$ be a finitely generated free group of rank at least two. Then the quotient of $\Comm(F)$ by $\AComm(F)$ is a free abelian group of infinite rank. 
\end{Theorem-intro}

Beyond naturalness of the definition, the subgroup  $\AComm(F)$ is of importance in $\Comm(F)$ because $\AComm(F)$ is equal to the monolith of $\Comm(F)$, and $\AComm(F)$ is a simple group  \cite[Theorem 5.4]{BELBRVW-comm-free}.  The combination of Theorem  \ref{thm-intro-comm-free} and \cite[Theorem 5.4]{BELBRVW-comm-free} yields:

\begin{Corollary-intro} \label{cor-intro-comm-free}
	Every proper quotient of $\Comm(F)$ is abelian.
\end{Corollary-intro}

\subsection{Approach and outline}

The proofs of Theorems  \ref{thm-intro-comme-tree-abelian-quotient} and \ref{thm-intro-comm-free} share a common context, that is described in Section \ref{sec-comm-abelian-quotient} below. There we define, for every group $\Gamma$, a canonical homomorphism from $\Comm(\Gamma)$ to a free abelian group. This construction uses an idea that originates from \cite{BELBRVW-comm-free} which consists in considering composition factors associated to domain and range in the context of virtual isomorphisms of $\Gamma$. This was used in \cite{BELBRVW-comm-free} to define a certain subgroup $\Comm_{SN}(\Gamma)$ of $\Comm(\Gamma)$, which was behind the aforementioned sufficient condition ensuring that the quotient of $\Comm(\Gamma)$ by $\AComm(\Gamma)$ is infinite (see the proof of Theorem 5.5 in \cite{BELBRVW-comm-free}). In Section \ref{sec-comm-abelian-quotient} we elaborate on this idea and show that it actually allows to define a homomorphism  $d^\Gamma : \Comm(\Gamma) \to \Z^{(\simple)}$, where $\Z^{(\simple)}$ denotes the free abelian group with basis the collection $\simple$ of isomorphism classes of finite simple groups. See Theorem \ref{thm-definition-morphism}. The kernel of this homomorphism is the subgroup $\Comm_{SN}(\Gamma)$ from \cite{BELBRVW-comm-free} (see Remark  \ref{rmk-CommSN}).

 In Section \ref{sec-comm-free} we study further this homomorphism in the case where $\Gamma$ is a finitely generated free group $F$. The main result of that section is Theorem \ref{thm-ker-CommF}, which asserts that the kernel of $d^F : \Comm(F) \to \Z^{(\simple)}$ is the subgroup $\AComm(F)$ of $\Comm(F)$. Together with the description of the image of $d^F : \Comm(F) \to \Z^{(\simple)}$ provided by Proposition \ref{prop-image-d^F}, this completes the proof of Theorem \ref{thm-intro-comm-free} from the introduction. In the proof of  Theorem \ref{thm-ker-CommF} we are naturally led to consider the following problem: given a finite simple group $S$, is the action of $\Aut(F)$ on the collection of normal subgroups $N$ of $F$ with $F/N \simeq S$  transitive ?  A conjecture due to Wiegold predicts that this is true for every $S$ and every free group $F$ of rank at least three. It is known for certain groups $S$, but it is open in general. We refer to the survey \cite{Lubotzky-survey-AutF-representations} for a detailed exposition on this subject and surrounding topics. In the proof of Theorem \ref{thm-ker-CommF} we take advantage of the fact that the setting where we need this transitivity property to hold  allows for passing to finite index subgroups of arbitrarily large index (which are free groups of arbitrarily large rank). This opens the possibility to invoke a substitute result, namely that the above transitivity property holds for every $S$ provided the rank of $F$ is large enough (Proposition  \ref{prop-transitive-AutF-large-rank}). Section \ref{sec-comm-free} terminates with a comparison between the abstract commensurators of a free group and a surface group. 
 
  Section \ref{sec-comm-tree} deals with commensurators of cocompact tree lattices. If $\Gamma$ is a cocompact lattice in $G = \Aut(T_d)$, then the natural homomorphism $\Comm_G(\Gamma) \to \Comm(\Gamma)$ is injective  \cite[B.7]{Bass-Kulk-90}. The group $\Gamma$ being virtually a free group,  $\Comm(\Gamma)$ is isomorphic to $\Comm(F)$.  The proof of Theorem  \ref{thm-intro-comme-tree-abelian-quotient} consists in considering the homomorphism $d^\Gamma : \Comm(\Gamma) \to \Z^{(\simple)}$, and showing that the image of the restriction of $d^\Gamma$ to $\Comm_G(\Gamma)$ is a subgroup of $\Z^{(\simple)}$ that is of infinite rank. This is carried out \S \ref{subsec-proof-comm-tree}.

\bigskip

\noindent \textbf{Acknowledgments}. We are grateful to Nir Lazarovich for a comment suggesting that the considerations in Theorem 5.5 in \cite{BELBRVW-comm-free} shall probably lead to non-simplicity of the commensurator of a cocompact tree lattice. Retrospectively, this comment served as a decisive impulse in the development of this work, including regarding Sections \ref{sec-comm-abelian-quotient} and \ref{sec-comm-free} which do not deal with commensurators of tree lattices. We also thank Alex Lubotzky for historical comments, and Francesco Fournier-Facio for pointing out \cite{Funar-Lochak} to our attention.

\bigskip

\noindent No AI tool was used at any stage of this work.

\section{Abstract commensurators} \label{sec-comm-abelian-quotient}

\subsection{Preliminaries} \label{subsec-prelim-comm}

Let $\Gamma$ be a group. An isomorphism $\varphi:A\to B$ between finite index subgroups $A,B$ of $\Gamma$ is called a \textbf{virtual isomorphism of $\Gamma$}.  The composition of two virtual isomorphism $\varphi:A\to B$ and $\psi:C\to D$ is defined as $ \varphi^{-1}(B \cap C) \to  \psi(B \cap C)$, $x \mapsto \psi (\varphi(x))$. Two virtual isomorphisms $\varphi:A\to B$ and $\varphi':A'\to B'$ are equivalent if there is a finite index subgroup $A'' \leq A \cap A'$ such that $\varphi$ and $\varphi'$ coincide on $A''$. We will write $[\varphi]$ for the equivalence class. Modulo this equivalence relation, virtual isomorphisms of $\Gamma$ form a group, denoted $\Comm(\Gamma)$, and called the abstract commensurator of $\Gamma$. In order to keep the exposition simple we choose not to consider any topology on $\Gamma$. But in case $\Gamma$ is a topological group (e.g.\ a profinite group), one defines $\Comm(\Gamma)$ similarly, except that finite index subgroups are replaced by finite index open subgroups, and virtual isomorphisms are required to be continuous. All the results of this section carry over to this setting.

If $A$ is a finite index subgroup of $\Gamma$, then any automorphism of $A$ defines a virtual isomorphism of $\Gamma$. The associated map $a_A: \Aut(A) \to \Comm(\Gamma)$ is a homomorphism. 

\begin{Definition} \label{def-AComm}
We denote by $\AComm(\Gamma)$ the subgroup of $\Comm(\Gamma)$ generated by the subgroups $a_A(\Aut(A))$, when $A$ ranges over finite index subgroups of $\Gamma$. 
\end{Definition}

\subsection{The homomorphism  $d^\Gamma : \Comm(\Gamma) \to \Z^{(\simple)}$}

In this subsection we elaborate on an idea from \cite{BELBRVW-comm-free}, which consists of considering composition factors associated to finite index subnormal subgroups in the setting of virtual isomorphisms. 

We denote by $\simple$ the collection of isomorphism classes of finite simple groups. Let $\Gamma$ be a group. If $A$ is a finite index subnormal subgroup of $\Gamma$, we call a finite series $\A = (A_\ell, \ldots, , A_0)$ a \textbf{composition series going from $A$ to $\Gamma$} if $A_\ell = A$, $A_0 = \Gamma$, each $A_{i+1}$ is a normal subgroup of $A_{i}$ and $A_{i} / A_{i+1} \in \simple$. The quotients $A_{i} / A_{i+1}$ are the \textbf{composition factors} of $\A$, and the integer $\ell$ is\textbf{ the composition length of $\A$}. For $S \in \simple$, we denote by $n_S^\Gamma(A)$ the number of composition factors of a composition series going from $A$ to $\Gamma$ that are isomorphic to $S$. By the Schreier refinement theorem, the composition length and the family of integers $\left\lbrace n_S^\Gamma(A)\right\rbrace_{S \in \simple} $ do not depend on the choice of a composition series going from $A$ to $\Gamma$. This common length is denoted $\ell^\Gamma(A)$ and is called \textbf{the composition length of $A$ in $\Gamma$}.

\begin{Lemma} \label{lem-formula-index}
If $A$ is a finite index subnormal subgroup of $\Gamma$, then  \[ (\Gamma : A) = \prod_{S \in \simple}  |S|^{n_S^\Gamma(A)}. \] 
\end{Lemma}

\begin{proof}
If $(A_\ell, \ldots, , A_0)$ is  a composition series going from $A$ to $\Gamma$ then  $(\Gamma : A) = \prod_{i=0}^{n-1}  (A_{i} : A_{i+1}) $. Each $A_{i} / A_{i+1}$ belongs to $\simple$, and given $S \in \simple$ there are $n_S^\Gamma(A)$ factors isomorphic to $S$. Whence the formula.
\end{proof}

\begin{Lemma} \label{lem-chasles}
Let $A$ be a finite index subnormal subgroup of $\Gamma$, and $B$  a finite index subnormal subgroup of $A$. Then  $n_S^\Gamma(B) = n_S^\Gamma(A) + n_S^A(B)$ for every $S \in \simple$. 
\end{Lemma}

\begin{proof}
This follows from the fact that the concatenation of a composition series going from $B$ to $A$ with a composition series going from $A$ to $\Gamma$ yields a composition series going from $B$ to $\Gamma$.
\end{proof}

\begin{Lemma} \label{lem-difference-independent-rep}
Let $A,B,C,D$ be finite index subnormal subgroups of $\Gamma$, and suppose $\varphi:A\to B$ and $\psi:C\to D$ are isomorphisms such that $[\varphi] = [\psi]$ in the group $\Comm(\Gamma)$. Then for every $S \in \simple$, one has  $n_S^\Gamma(B) - n_S^\Gamma(A) = n_S^\Gamma(D) - n_S^\Gamma(C)$. 
\end{Lemma}

\begin{proof}
The assumption  $[\varphi] = [\psi]$ means that there is a finite index subgroup $E$ of $A \cap C$ such that $\varphi, \psi$ coincide on $E$. Upon passing to the normal core of $E$ in $\Gamma$, one can assume $E$ is normal in $\Gamma$. Set $F = \varphi(E) = \psi(E)$. Since $E$ is normal in $A$, the subgroup $F =  \varphi(E)$ is normal in $B = \varphi(A)$. Hence $F$ is subnormal in $\Gamma$. 

Using Lemma \ref{lem-chasles} for the two sequences of inclusions $F \leq B \leq \Gamma$ and $E \leq A \leq \Gamma$ we see that \[  n_S^\Gamma(F) = n_S^\Gamma(B) + n_S^B(F) \] and  \[ n_S^\Gamma(E) = n_S^\Gamma(A) + n_S^A(E).\] Taking the difference, we obtain  \[ n_S^\Gamma(F) - n_S^\Gamma(E) = n_S^\Gamma(B) + n_S^B(F) - (n_S^\Gamma(A) + n_S^A(E)).\] Now since $\varphi$ is an isomorphism from $A$ to $B$ sending $E$ to $F$, one has $n_S^A(E) = n_S^B(F)$. Hence those terms cancel in the previous equality, and we obtain $n_S^\Gamma(B) - n_S^\Gamma(A) = n_S^\Gamma(F) - n_S^\Gamma(E)$. Since the situation is symmetric in $\varphi, \psi$, one also has $n_S^\Gamma(D) - n_S^\Gamma(C) = n_S^\Gamma(F) - n_S^\Gamma(E)$. Whence the conclusion. 
\end{proof}

The following lemma is classical. 

\begin{Lemma}  \label{lem-existence-subnormal-rep}
	Let $A,B$ be finite index subgroups of $\Gamma$ and let $\varphi: A \to B$ be an isomorphism. Then there are finite index subgroup $A_1,A_2$ of $A$ and $B_1$ of $B$ such that  $A_1 \lhd A_2 \lhd \Gamma$, $B_1 \lhd \Gamma$ and $B_1 = \varphi(A_1)$.
\end{Lemma}

\begin{proof}
	Let $A_2$ denote the normal core of $A$ in $\Gamma$, and $B_2 = \varphi(A_2)$. Now let $B_1$ denote the  normal core of $B_2$ in $\Gamma$, and let $A_1 = \varphi^{-1}(B_1)$. By construction one has $A_1 \lhd A_2 \lhd \Gamma$, $B_1 \lhd \Gamma$ and $B_1 = \varphi(A_1)$.
\end{proof}

An isomorphism  $\varphi: A \to B$ between finite index subgroups of $\Gamma$ will be called \textbf{$\Gamma$-subnormal} if both $A$ and $B$ are subnormal subgroups of $\Gamma$. Lemma \ref{lem-existence-subnormal-rep} implies every element of the commensurator $\Comm(\Gamma)$ admits a $\Gamma$-subnormal representative. 

We denote by $\Z^{(\simple)}$ the free abelian group with basis $\simple$, viewed as the group of finitely supported functions $\simple \to \Z$. For $S \in \simple$ we denote by $\delta_S$ the Dirac function at $S$. 

\begin{Theorem} \label{thm-definition-morphism}
	The map $d^\Gamma : \Comm(\Gamma) \to \Z^{(\simple)}$ defined by \[c = [\varphi: A \to B] \mapsto d^\Gamma(c) =  \sum_{S \in \simple} (n_S^\Gamma(B) - n_S^\Gamma(A)) \delta_S,\] where $\varphi: A \to B$ is a $\Gamma$-subnormal representative of $c$, is well-defined and is a group homomorphism. 
\end{Theorem}

Before proving the theorem, we introduce the following notation: 

\begin{Definition} \label{defi-map-dS}
For  $S \in \simple$, we consider the map $d^\Gamma_S : \Comm(\Gamma) \to \Z$ defined by \[c = [\varphi: A \to B] \mapsto d^\Gamma_S(c) =  n_S^\Gamma(B) - n_S^\Gamma(A), \] where $\varphi: A \to B$ is a $\Gamma$-subnormal representative of $c$.
\end{Definition}

For every $c \in  \Comm(\Gamma)$, one has \[ d^\Gamma(c) =  \sum_{S \in \simple} d^\Gamma_S(c) \delta_S,   \] so the statement of Theorem \ref{thm-definition-morphism} will follow from the following proposition:

\begin{Proposition} \label{prop-one-prime-morphism}
	For every  $S \in \simple$, the map $d^\Gamma_S : \Comm(\Gamma) \to \Z$ is well-defined and is a group homomorphism. 
\end{Proposition}

\begin{proof}
First recall  every element of $\Comm(\Gamma)$  admits a $\Gamma$-subnormal representative by Lemma \ref{lem-existence-subnormal-rep}. Moreover Lemma  \ref{lem-difference-independent-rep} ensures that, given  $c  \in \Comm(\Gamma)$ such that $c$ admits $\varphi: A \to B$ as a $\Gamma$-subnormal representative, $d^\Gamma_S(c) =  n_S^\Gamma(B) - n_S^\Gamma(A)$ does not depend on the choice of the $\Gamma$-subnormal representative of $c$. This shows $d^\Gamma_S$ is well-defined. Let us prove it is a homomorphism. Let $c_1, c_2 \in \Comm(\Gamma)$, and choose $\Gamma$-subnormal representatives $\psi:C\to D$ and $\varphi:A\to B$  respectively of $c_1$ and $c_2$. Set $E = \varphi^{-1}(B \cap C)$ and $F = \psi(B \cap C)$. The subgroup $B \cap C$ is subnormal in $\Gamma$ as an intersection of two subnormal subgroups. In particular $B \cap C$ is subnormal in $B$. Therefore $E$ is subnormal in $A = \varphi^{-1}(B)$, and hence $E$ is subnormal in $\Gamma$. Similarly $F$ is subnormal in $\Gamma$. The restriction of $\psi \circ \varphi$ to $E$ induces an isomorphism from $E$ to $F$, which by definition of the composition in $\Comm(\Gamma)$ is a representative of $c_1 c_2$. Therefore one has \[d^\Gamma_S(c_1 c_2) = n_S^\Gamma(F) - n_S^\Gamma(E) = n_S^\Gamma(F) - n_S^\Gamma(B \cap C) + n_S^\Gamma(B \cap C) - n_S^\Gamma(E). \] Now by construction $\psi_{|B \cap C}: B \cap C \to F$ and $\varphi_{|E}: E \to B \cap C$ are $\Gamma$-subnormal representatives of $c_1$ and $c_2$ respectively. Hence $ n_S^\Gamma(F) - n_S^\Gamma(B \cap C) = d^\Gamma_S(c_1)$ and $n_S^\Gamma(B \cap C) - n_S^\Gamma(E) = d^\Gamma_S(c_2)$.  So all together we have $d_S(c_1 c_2) = d_S(c_1) + d_S(c_2)$, as desired. 
\end{proof}

\begin{Remark} \label{rmk-CommSN}
The kernel of $d^\Gamma$ is the set of elements of $\Comm(\Gamma)$ that admit a $\Gamma$-subnormal representative $\varphi: A \to B$ such that $n_S^\Gamma(B) = n_S^\Gamma(A)$ for every $S \in \simple$. This means that the sets of composition factors appearing in composition series going from $A$ to $\Gamma$ and $B$ to $\Gamma$ are the same (taking multiplicity into account). That set was denoted $\Comm_{SN}(\Gamma)$ in \cite{BELBRVW-comm-free}, and it has already been showed there that it is a subgroup of $\Comm(\Gamma)$.
\end{Remark}

\begin{Lemma} \label{lem-subnormal-rep-inside-subnormal}
	Let $A,B,C$ be finite index subnormal subgroups of $\Gamma$ and let $\varphi: A \to B$ be an isomorphism. Then there are finite index subgroups $A' \leq A$ and  $B' \leq B$ such that $A',B'$ are contained in $C$,  $A',B'$ are subnormal in $C$, and $B' = \varphi(A')$.
\end{Lemma}

\begin{proof}
The subgroups $A \cap C$ and $B \cap C$ are subnormal respectively in $A$ and $B$. Hence $A' := A \cap C \cap \varphi^{-1}(B \cap C)$ is subnormal in $A$. In particular $A'$ is subnormal in $A \cap C$, and hence in $C$ since $A \cap C$ is subnormal  in $C$. The argument to see that $B' := \varphi(A')$ is subnormal in $C$ is similar.  
\end{proof}

If $\Lambda$ is a finite index subgroup of $\Gamma$, then any virtual isomorphism of $\Lambda$ is also a virtual isomorphism of $\Gamma$. This induces a homomorphism $i_{\Lambda,\Gamma} : \Comm(\Lambda) \to \Comm(\Gamma)$, which is easily seen to be an isomorphism.

\begin{Proposition} \label{prop-compatibility-hom-finite-index}
	Let $\Lambda$ be a finite index subgroup of $\Gamma$, and let $i_{\Lambda,\Gamma} : \Comm(\Lambda) \to \Comm(\Gamma)$ be the isomorphism induced by the inclusion of $\Lambda$ in $\Gamma$. Then $d^\Gamma \circ i_{\Lambda,\Gamma}= d^\Lambda$. 
\end{Proposition}

\begin{proof}
Denote by $\Lambda'$  the normal core of $\Lambda$ in $\Gamma$.  Take $c  \in \Comm(\Lambda)$, and let $\varphi: A \to B$ be a $\Lambda$-subnormal representative of $c$ in $\Comm(\Lambda)$. According to Lemma \ref{lem-subnormal-rep-inside-subnormal} one can assume $A,B$ are contained in $\Lambda'$. In particular $A,B$ are subnormal in $\Gamma$, and $\varphi: A \to B$ be a $\Gamma$-subnormal representative of $c$ in $\Comm(\Gamma)$. Take $S \in \simple$. We have $n_S^\Gamma(A) = n_S^\Gamma(\Lambda') + n_S^{\Lambda'}(A)$ and $n_S^\Gamma(B) = n_S^\Gamma(\Lambda') + n_S^{\Lambda'}(B)$ by Lemma \ref{lem-chasles}. Hence  $n_S^\Gamma(B) - n_S^\Gamma(A) = n_S^{\Lambda'}(B) -  n_S^{\Lambda'}(A)$. Since $S$ is arbitrary, this proves the statement.
\end{proof}

The following result was proven in \cite{BELBRVW-comm-free} under the additional assumption that $\Gamma$ is finitely generated. 

\begin{Proposition} \label{prop-AComm-contained-kernel}
The subgroup $\AComm(\Gamma)$ is contained in the kernel of $d^\Gamma$. 
\end{Proposition}

\begin{proof}
Let $\Lambda$ be a finite index subgroup of $\Gamma$. Take $\varphi \in \Aut(\Lambda)$. By definition $\varphi: \Lambda \to \Lambda$ is a $\Lambda$-subnormal representative of $[\varphi]$ in $\Comm(\Lambda)$. Since $n^\Lambda_S(\Lambda) = 0$ for every $S \in \simple$, the element $[\varphi]$ of $\Comm(\Lambda)$ belongs to the kernel of $d^\Lambda$. Proposition \ref{prop-compatibility-hom-finite-index} ensures that the element of $ \Comm(\Gamma)$ defined by $\varphi$ belongs to the kernel of $d^\Gamma$. Since $\Lambda$ was an arbitrary finite index subgroup of $\Gamma$, this shows $\AComm(\Gamma)$ is contained in the kernel of $d^\Gamma$. 
\end{proof}

\subsection{On the image of the homomorphism  $d^\Gamma$}

In this subsection we collect two basic observations regarding the following homomorphism: 

\begin{Definition}
	We denote by  $\pi: \Z^{(\simple)} \to \Q^\times_{>0}$ the  homomorphism defined by $(x_S)_\simple \mapsto \prod_\simple  |S|^{x_S}$. 
\end{Definition}

\begin{Lemma} \label{lem-im-contained-ker-same-index}
If $c \in \Comm(\Gamma)$ admits $\varphi: A \to B$ as a $\Gamma$-subnormal representative, then \[\pi \circ d^\Gamma (c) = \frac{(\Gamma : B)}{(\Gamma : A)}. \] In particular if $\Gamma$ has the property that any two isomorphic finite index subgroups of $\Gamma$  have the same index, then the image of $d^\Gamma$ is contained in the kernel of $\pi$.
\end{Lemma}

\begin{proof}
One has \[\pi \circ d^\Gamma (c) = \prod_\simple  |S|^{n_S^\Gamma(B) - n_S^\Gamma(A)} = \frac{\prod_\simple  |S|^{n_S^\Gamma(B)}}{\prod_\simple  |S|^{n_S^\Gamma(A)}} = \frac{(\Gamma : B)}{(\Gamma : A)}\] by Lemma  \ref{lem-formula-index}. The second assertion is clearly consequence of the first.
\end{proof}

Recall that for $S \in \simple$  we denote by $\delta_S \in \Z^{(\simple)}$ the Dirac function at $S$. When $S = C_p$ is cyclic of order $p$ then we will write $\delta_p$ instead of $\delta_{C_p}$. We denote by $\simpleab$ the collection of those finite simple groups that are cyclic of primer order. We identify $\Z^{(\simpleab)}$ as the subgroup of $\Z^{(\simple)}$ generated by the elements $\delta_p$ when $p$ varies among primes. 

\begin{Definition} \label{def-notation-ker-pi}
For an integer $n \geq 2$ with prime decomposition $n=p_1^{\alpha_1} \cdots p_r^{\alpha_r}$, we write $g_n = \sum \alpha_i \delta_{p_i} \in \Z^{(\simple)}$. For $S \in \simple$ with $S$ non-abelian, we denote $f_S = \delta_S - g_{|S|}$. 
\end{Definition}

\begin{Lemma} \label{lem-generation-kernel-abelian}
	The following hold: \begin{enumerate}
		\item The restriction of $\pi$ to  $\Z^{(\simpleab)}$ induces an isomorphism $\Z^{(\simpleab)} \to \Q^\times_{>0}$. 
		\item $\ker(\pi)$ is generated by the elements $f_S$ when $S$ varies among non-abelian finite simple groups. 
	\end{enumerate}
\end{Lemma}

\begin{proof}
	The first point is clear. For the second, one has $\pi( \delta_S) = |S|$ and $  \pi(g_{|S|})  = |S|$, so $f_S$ indeed lies in $\ker(\pi)$. Let $A$ denote the subgroup generated by $f_S$ when $S$ varies among non-abelian finite simple groups. Since $f_S \in \delta_S + \Z^{(\simpleab)}$, one has $\Z^{(\simple)} = A + \Z^{(\simpleab)}$, and hence $\ker(\pi) = A + (\ker(\pi) \cap \Z^{(\simpleab)})$. Now $\ker(\pi) \cap \Z^{(\simpleab)} = 0$ by the first point, so $\ker(\pi) = A$. 
\end{proof}

\subsection{Commensurators of profinite completions}

As mentioned in \S \ref{subsec-prelim-comm}, both the group of commensurators $ \Comm(U)$ and the homomorphism $d^U : \Comm(U) \to \Z^{(\simple)}$ can be defined when $U$ is a profinite group. Although this will not be used later on in this paper, in this subsection we make the observation that when passing from a group $\Gamma$ to its profinite completion $\widehat{\Gamma}$, the expected compatibility between $d^\Gamma : \Comm(\Gamma) \to \Z^{(\simple)}$ and $d^{\widehat{\Gamma}} : \Comm(\widehat{\Gamma}) \to \Z^{(\simple)}$ holds true. 

So let $\Gamma$ be a group, and denote by $\iota: \Gamma \to \widehat{\Gamma}$ the natural homomorphism from $\Gamma$ to $\widehat{\Gamma}$. The mapping that assigns to each finite index subgroup $A$ of $\Gamma$ its image closure $\overline{\iota(A)}$ in $\widehat{\Gamma}$ is a bijection between finite index subgroups of $\Gamma$ and open subgroups of $\widehat{\Gamma}$. Moreover the natural homomorphism $\widehat{A} \to \overline{\iota(A)}$ is an isomorphism  \cite[Lemma 3.2.6]{RZ-book-profinite} In the sequel we always see $\widehat{A}$ as an open subgroup of $\widehat{\Gamma}$.  Any virtual isomorphism $\varphi: A \to B$ of $\Gamma$ induces an isomorphism $\widehat{\varphi}: \widehat{A} \to \widehat{B}$, which is therefore viewed as a virtual isomorphism of $\widehat{\Gamma}$. The map $\varphi \mapsto \widehat{\varphi}$ descends to a map $\eta_\Gamma: \Comm(\Gamma) \to \Comm(\widehat{\Gamma}) $. Moreover this map is a homomorphism. 

\begin{Proposition} \label{prop-completion-same-det-hom}
The diagram 
 \[ \begin{tikzcd}
\Comm(\Gamma) \arrow [r, "\eta_\Gamma"]  \arrow [dr, "d^\Gamma"] & \Comm(\widehat{\Gamma}) \arrow [d, "d^{\widehat{\Gamma}}"] \\
 &  \Z^{(\simple)}
\end{tikzcd} \]
commutes.
\end{Proposition}

\begin{proof}
The mapping $A \mapsto \widehat{A}$ from finite index subgroups of $\Gamma$ to open subgroups of $\widehat{\Gamma}$ sends subnormal subgroups to subnormal subgroups. Moreover any composition series going from $A$ to $\Gamma$ is sent to a composition series going from $\widehat{A}$ to $\widehat{\Gamma}$ with isomorphic composition factors. Hence if $c$ is an element of $\Comm(\Gamma)$ with subnormal representative $\varphi: A \to B$, then $\widehat{\varphi}: \widehat{A} \to \widehat{B}$ is a subnormal representative of $\eta_\Gamma(c)$, and for any $S \in \simple$ we have $n_S^\Gamma(A) = n_S^{\widehat{\Gamma}}(\widehat{A})$ and $n_S^\Gamma(B) = n_S^{\widehat{\Gamma}}(\widehat{B})$. The statement follows. 
\end{proof}

\subsection{Examples}

We examine what is the homomorphism $d^\Gamma$ is some familiar examples. 

\subsubsection{$ \Gamma = \Z^n$, $n \geq 1$}
Then $\Comm(\Gamma) \simeq \GL(n,\Q)$. Clearly the only finite simple groups involved are the cyclic groups of prime order. Hence $d^\Gamma$ takes values in $\Z^{(\simpleab)}$. The subgroup $\SL(n,\Q)$ lies in the kernel of $d^\Gamma$ (as can be check directly, or it also follows from the fact that $\SL(n,\Q)$ is perfect). Moreover for a prime $p$, the diagonal matrix $D$ with entry $p$ at coordinate $i$ and $1$ at other diagonal entries provides an isomorphism between $A = \Z^n$ and $B = \Z^{i-1} \times p \Z \times \Z^{n-i}$. So $d^\Gamma_{p}(D) = 1$ and $d^\Gamma_{q}(D) = 0$ for every prime $q \neq p$. So, under the  identification between $\Z^{(\simpleab)}$ and $\Q_{> 0}^\times$ via $\pi$, one has $d^\Gamma(D) = |\det(D)|$. It  follows that $d^\Gamma$ and $ |\det|$ coincide on $\SL(n,\Q)$  and on all diagonal matrices. Therefore they coincide on the entire $\GL(n,\Q)$. So $d^\Gamma =  |\det|$.

\subsubsection{$ \Gamma = \PSL(n,\Z)$, $n \geq 3$} As follows from Mostow rigidity and \cite[Theorem 2]{Borel-density}, the abstract commensurator of $\Gamma$ is $\Comm(\Gamma) \simeq \PGL(n,\Q) \rtimes C_2$, where $C_2$ acts via transpose-inverse. The determinant induces a surjective homomorphism $\PGL(n,\Q) \to \Q^\times / (\Q^\times)^n$, and $\Q^\times / (\Q^\times)^n$ is the abelianization of $\PGL(n,\Q)$. So every abelian quotient of $\PGL(n,\Q)$ is torsion, and similarly for $\Comm(\Gamma)$. Since $d^\Gamma$ always takes values in a free abelian group, $d^\Gamma$ is trivial here. What happens for $n=2$, which corresponds to the case of a free group since $\PSL(2,\Z)$ is virtually free, is the subject of the next section.

\section{On the abstract commensurator of a finitely generated free group} \label{sec-comm-free}

In this section we consider the situation of a finitely generated free group of rank at least two, denoted by $F$.

\subsection{The image of $d^F : \Comm(F) \to \Z^{(\simple)}$}

\begin{Proposition} \label{prop-image-d^F}
	The image of $d^F : \Comm(F) \to \Z^{(\simple)}$ is \[ d^F (\Comm(F)) = \left\lbrace (x_S)_\simple \in \Z^{(\simple)} \, : \, \prod_\simple  |S|^{x_S}= 1 \right\rbrace. \] In particular it is a free abelian group of infinite rank.
\end{Proposition}

\begin{proof}
	Using previous notation, the statement is that the image of $d^F : \Comm(F) \to \Z^{(\simple)}$ is exactly the kernel of the homomorphism $\pi: \Z^{(\simple)} \to \Q^\times_{>0}$, $(x_S)_\simple \mapsto \prod_\simple  |S|^{x_S}$. By  the Nielsen-Schreier formula, two finite index subgroups of $F$ are isomorphic if and only if they have the same index. Hence by Lemma  \ref{lem-im-contained-ker-same-index}  the image of $d^F$ is contained in $\ker(\pi)$. Let $S$ be a non-abelian finite simple group, and let $n$ be the cardinality of $S$. It is known that every finite simple group is generated by two elements. Hence one can find a surjective homomorphism $F \to S$. Let $B$ be the kernel of this homomorphism. Choose also a surjective homomorphism $F \to C_n$ from $F$ to the cyclic group of order $n$, and denote by $A$ its kernel. Since $A,B$ have the same index in $F$, one can choose an isomorphism $\varphi: A \to B$. Let $c$ be the corresponding element of $\Comm(F)$. Using notation from Definition \ref{def-notation-ker-pi}, by construction of $c$ one has $d^F(c) = \delta_S - g_{|S|} = f_S$. Lemma \ref{lem-generation-kernel-abelian} then ensures the image of $d^F$ equals the entire $\ker(\pi)$.
\end{proof}

\subsection{The kernel of $d^F : \Comm(F) \to \Z^{(\simple)}$}

The goal of this subsection is to prove Theorem  \ref{thm-ker-CommF}. For, the main intermediate result is Theorem \ref{thm-caract-conj-AComm}, and its proof is the core of this subsection.

\begin{Definition}
For a finite group $G$, we denote by $ \mathcal{K}(F; G)$ the set of normal subgroups $N$ of $F$ such that $F/N \simeq G$. 
\end{Definition}

We recall the following interpretation of $\mathcal{K}(F; G)$. Let $E(F; G)$ denote the collection of epimorphisms  $F \to G$. The group $\Aut(G)$ acts on $E(F; G)$ by post-composition, and the group $\Aut(F)$ acts on $E(F; G)$ by pre-composition. These actions commute. The map  $E(F; G) \to \mathcal{K}(F; G)$ that associates to an epimorphism $\pi: F \to G$ its kernel $\ker(\pi)$ is surjective, and factors through a map $E(F; G) / \Aut(G) \to \mathcal{K}(F;G)$. Moreover any two epimorphisms $F \to G$ have the same kernel only if they belong to the same $\Aut(G)$-orbit, so  the map $E(F; G) / \Aut(G) \to \mathcal{K}(F; G)$ is bijective. Note that this map is $\Aut(F)$-equivariant. 

If $G$ is a finite group, we denote by $d(G)$ the minimal number of generators of $G$. We also write $\mu(G) = \max_\Sigma |\Sigma|$, where the maximum is taken over all minimal generating subsets of $G$. We will use the following result, which is \cite[Proposition 3.1]{Lubotzky-survey-AutF-representations} (see also \cite[Theorem 3]{Gilman-77-AutF}). 

\begin{Proposition} \label{prop-transitive-AutF-large-rank}
Let $G$ be a finite group. If the rank of $F$ is larger than $d(G) + \mu(G)$, then the $\Aut(F)$-action on $E(F; G)$ is transitive. In particular the $\Aut(F)$-action on $\mathcal{K}(F; G)$ is transitive.
\end{Proposition}

It will be convenient to use the following notation: 

\begin{Notation}
	If $S \in \simple$ and if $A$ is a subgroup of $B$, we write $A \lhd_{S} B $ if $A$ is normal in $B$ and $B/A \simeq S$.
\end{Notation}

\begin{Lemma} \label{lem-switch-comp-factors}
	Let $S_1, S_2 \in \simple$ such that $S_1,S_2$ are not isomorphic. Suppose $A_1,A_2$ are subgroups of $F$ such that $A_2 \lhd_{S_1} A_1 \lhd_{S_2} F$. Then there are subgroups $B_1,B_2$ of $F$ such that : \begin{enumerate}
		\item $B_2 \lhd_{S_1} A_1 \lhd_{S_2} F$; 
		\item $B_2 \lhd_{S_2} B_1 \lhd_{S_1} F$. 
	\end{enumerate}
\end{Lemma}

\begin{proof}
	Since $F$ is a free group, there exists a normal subgroup $B_1$ of $F$ such that $F/B_1 \simeq {S_1}$. Consider the normal subgroup $A_1B_1$ of $F$. Since $A_1$ is a maximal normal subgroup of $F$, we have $A_1B_1 = F$ or $A_1B_1 = A_1$. Assume the second possibility holds.  This means $B_1 \leq A_1$. Since $B_1$ is a maximal normal subgroup of $F$, we infer $A_1 = B_1$, which is impossible since we assume $S_1,S_2$ are not isomorphic. Hence $A_1B_1 = F$. Set $B_2 = A_1 \cap B_1$. Equality $A_1B_1 = F$ implies $A_1 / B_2 \simeq S_1$ and  $B_1 / B_2 \simeq S_2$. Hence $B_1,B_2$ satisfy the requirements. 
\end{proof}

If $A$ is a finite index subgroup of $F$, the natural homomorphism $a_A: \Aut(A) \to \Comm(F)$ is injective. For simplicity we will write $\Aut(A)$ instead of $a_A(\Aut(A))$ for the image of $\Aut(A)$ in $ \Comm(F)$. Also we will identify $A$ with the subgroup of $\Aut(A)$ consisting of inner automorphisms. In particular we will always view $A$ as a subgroup of $\Comm(F)$. The subgroup $\Aut(A)$ is exactly the normalizer of $A$ in $\Comm(F)$.


\begin{Theorem} \label{thm-caract-conj-AComm}
Suppose $A,B$ are finite index subnormal subgroups of $F$. The following are equivalent: \begin{enumerate}
		\item \label{item-A-B-conj} $A,B$ are conjugate in $\AComm(F)$;
		\item \label{item-A-B-sames-factors} $n^F_S(A) = n^F_S(B)$ for every $S \in \simple$. 
	\end{enumerate}
\end{Theorem}

\begin{proof}
	We start with (\ref{item-A-B-conj}) $\implies$ (\ref{item-A-B-sames-factors}). Take $c \in \AComm(F)$ such that $B = cA c^{-1}$. Then $\varphi: A \to B$, $a \mapsto c a c^{-1}$, is a representative of $c$. Since $A,B$ are subnormal in $F$, it is a subnormal representative. We know from Proposition  \ref{prop-AComm-contained-kernel} that $c$ belongs to the kernel of $d^F : \Comm(F) \to \Z^{(\simple)}$. Therefore $n^F_S(A) = n^F_S(B)$ for every $S \in \simple$, as desired. 
	
(\ref{item-A-B-sames-factors})  $\implies$ (\ref{item-A-B-conj}) is the more difficult part. We break the proof into several steps. 
	
\begin{Step} \label{step-length-one-case}
	(\ref{item-A-B-sames-factors})  $\implies$ (\ref{item-A-B-conj}) is true when $A,B$ verify $\ell^F(A) = \ell^F(B) = 1$. 
\end{Step}

\begin{proof}
Given $S \in \simple$ and $A,B$ two subgroups of $F$ with $A \lhd_{S} F$ and $B \lhd_{S} F$, we want to show that $A,B$ are conjugate in $\AComm(F)$. Since $A,B$ have the same index in $F$, they are free groups of the same rank. Let $\varphi: A \to B$ be an isomorphism. We have $[\varphi] A [\varphi]^{-1} = B$.   We show that $[\varphi]$ belongs to $\AComm(F)$.
 
  If there exists a finite index subgroup $C$ of $A$ such that $C = \varphi(C)$, then $[\varphi]$ belongs to $\Aut(C)$ in $\Comm(F)$. Since $\Aut(C)$ is contained in $\AComm(F)$, in that case we are done. Henceforth we assume that there is no  finite index subgroup $C$ of $A$ such that $C = \varphi(C)$. We define inductively  decreasing sequences $(A_n)_{n \geq 1}, (B_n)_{n \geq 0}$ of finite index subgroups of $F$ with the following properties: \begin{itemize}
  	\item $A_{n+1} \lhd_{S} B_n$ and $B_{n+1} \lhd_{S} B_n$; 
  	\item the restriction of $\varphi$ to $A_{n+1}$ induces an isomorphism $A_{n+1} \to B_{n+1}$.
  \end{itemize} 
By assumption $B_0 = F$, $A_1 = A$ and $B_1 = B$ satisfy the required properties. Suppose $A_1, \ldots, A_n$ and $B_0, \ldots, B_n$ have already been defined. The subgroups $A_n$ and $B_n$ are maximal normal subgroups of $B_{n-1}$. Hence if $A_n B_n \neq B_{n-1}$ then $A_n = B_n$. Since the restriction of $\varphi$ to $A_{n}$ induces an isomorphism $A_{n} \to B_{n}$, this contradicts the present assumption that there is no finite index subgroup $C$ such that $C = \varphi(C)$. Hence we must have $A_n B_n = B_{n-1}$. We set $A_{n+1} = A_n \cap B_n$ and $B_{n+1} = \varphi(A_{n+1})$. Equality $A_n B_n = B_{n-1}$ ensures $B_{n-1} / B_n \simeq A_n / A_{n+1}$. Therefore $A_n / A_{n+1} \simeq S$. Since $B_{n+1} = \varphi(A_{n+1})$ and $B_{n} = \varphi(A_{n})$, we infer $B_{n+1}$ is normal in $B_n$ and  $B_n / B_{n+1} \simeq A_n / A_{n+1}  \simeq S$. On the other hand the equality $A_n B_n = B_{n-1}$ also ensures $B_{n-1} / A_n \simeq  B_{n} / A_{n+1}$, and therefore $B_{n} / A_{n+1} \simeq S$. Therefore $A_{n+1}, B_{n+1}$ satisfy the required properties. 

Since the sequence $(B_n)$ is properly decreasing, one can find $n$ such that the rank of $B_n$ is larger than the integer appearing in the hypothesis of Proposition  \ref{prop-transitive-AutF-large-rank}. Since $A_{n+1}, B_{n+1} \in \mathcal{K}(B_n; S)$, by the proposition one can find $\alpha \in \Aut(B_n)$ such that $\alpha(B_{n+1}) = A_{n+1}$. Since the restriction of $\varphi$ to $A_{n+1}$ induces an isomorphism $A_{n+1} \to B_{n+1}$, we have $\alpha [\varphi] \in \Aut(A_{n+1})$. Since $\Aut(B_n)$ and $\Aut(A_{n+1})$ are contained in $\AComm(F)$, this shows $[\varphi] \in \AComm(F)$, as desired. 
\end{proof}

\begin{Step} \label{step-move-position-one}
Let $S \in \simple$, and let $A$ be a finite index subnormal subgroup of $F$ such that $n^F_S(A) > 0$. Then there exists a finite index subnormal subgroup $B$ of $F$ such that: \begin{enumerate}
	\item $A$ and $B$ are conjugate in $\AComm(F)$; 
	\item there exists a composition series $(B_\ell, \ldots, B_0)$  going from $B$ to $F =B_0$ such that $B_1 \lhd_{S} F$.
\end{enumerate}
\end{Step}

\begin{proof}
Let   $\A = (A_\ell , \ldots, A_0)$ be a composition series going from $A$ to $\Gamma$. By assumption one composition factor is isomorphic to $S$. Let $m(\A;S)$ be the smallest integer $i \geq 0$ such that $ A_{i+1}  \lhd_{S} A_i$. Let also $m(A;S) = \min_\A m(\A;S)$, where the minimum is taken over all composition series going from $A$ to $\Gamma$. We argue by induction on $m(A;S)$. 

The statement is true when $m(A;S)=0$, as we can just take $B=A$. Assume now $k \geq 0$ is such that the statement is true for every $A$ such that $m(A;S) \leq k$, and let $A$ such that $m(A;S) = k+1$. Let $\A = (A_\ell , \ldots, A_0)$ be a composition series going from $A$ to $\Gamma$ such that $A_{k+2} \lhd_{S} A_{k+1}$ and $A_i / A_{i+1}$ is not isomorphic to $S$ for $i \leq k$. In particular $A_k / A_{k+1}$ is not isomorphic to $S$. So we have \[A_{k+2} \lhd_{S} A_{k+1} \lhd_{S'} A_{k} \] with $S' \neq S$. We apply Lemma \ref{lem-switch-comp-factors} (inside the finitely generated free group $A_{k}$). We deduce there exist $B_{k+2}, B_{k+1}$ such that $B_{k+2} \lhd_{S} A_{k+1} \lhd_{S'} A_{k}$ and $B_{k+2} \lhd_{S'} B_{k+1} \lhd_{S} A_{k}$. Since $A_{k+2} \lhd_{S} A_{k+1}$ and $B_{k+2} \lhd_{S} A_{k+1}$, by Step \ref{step-length-one-case} above (applied inside $A_{k+1}$) we deduce that $A_{k+2}$ and $ B_{k+2}$ are conjugate in $\AComm(A_{k+1})$. Now $A_{k+1}$ is a finite index subgroup of $F$, so $\AComm(A_{k+1}) = \AComm(F)$. So we can find $c \in  \AComm(F)$ such that  $cA_{k+2}c^{-1} = B_{k+2}$. Set $A' := cAc^{-1}$. We therefore have a composition series $\A' = (cA_\ell c^{-1}, \ldots, cA_{k+2}c^{-1}, B_{k+1}, A_{k}, \ldots, A_0)$ going from $A'$ to $\Gamma$. By construction we have $m(\A;S) = k$. In particular  $m(A';S) \leq k$. By the induction hypothesis applied to $A'$ and since $A$ and $A'$ are conjugate in $\AComm(F)$, this terminates the proof.
\end{proof}

\begin{Step}
Conclusion of the proof of (\ref{item-A-B-sames-factors})  $\implies$ (\ref{item-A-B-conj}) in Theorem \ref{thm-caract-conj-AComm}. 
\end{Step}

\noindent
 We take $A,B$ finite index subnormal subgroups of $F$ with $n^F_S(A) = n^F_S(B)$ for every $S \in \simple$. We wish to show $A,B$ are conjugate in $\AComm(F)$. We argue by induction on the composition length $\ell^F(A) = \ell^F(B)$. The case $\ell^F(A) = 1$ is exactly the content of Step \ref{step-length-one-case}. Now assume $\ell^F(A) = \ell + 1$ for some $\ell \geq 1$. Pick a composition series $(B_\ell, \ldots, B_0)$  going from $B$ to $F$, and let $S \in \simple$ such that $B_1 \lhd_{S} B_0$. By the assumption on $A,B$, we have $n^F_S(A) > 0$. We apply Step \ref{step-move-position-one}. By the conclusion of  Step \ref{step-move-position-one} and by the implication (\ref{item-A-B-conj}) $\implies$ (\ref{item-A-B-sames-factors}) of Theorem \ref{thm-caract-conj-AComm}, we deduce that upon conjugating $A$ by an element of $\AComm(F)$, we can assume that there exists a composition series $(A_\ell, \ldots, A_0)$  going from $A$ to $F$ such that $A_1 \lhd_{S} A_0$. Since $B_1 \lhd_{S} F$ and $A_1 \lhd_{S} F$, by Step \ref{step-length-one-case} $A_1$ and $B_1$ are conjugate in $\AComm(F)$. Hence upon conjugating we are reduced to the situation where $A_1 = B_1$. We have $n^{A_1}_{S}(A) = n^F_{S}(A) - 1$ and $n^{A_1}_{S'}(A) = n^F_{S'}(A)$ for every $S'$ different from $S$. And similarly for $B$. Therefore $n^{A_1}_{S'}(A) = n^{A_1}_{S'}(B)$ for every $S' \in \simple$. Since $\ell^{A_1}(A) =  \ell^F(A) - 1$, we can apply the induction hypothesis to the subgroups $A,B$ inside the ambient group $A_1$. This terminates the proof. 
\end{proof}

\begin{Theorem} \label{thm-ker-CommF}
	The kernel of $d^F : \Comm(F) \to \Z^{(\simple)}$ is the subgroup $\AComm(F)$ of $\Comm(F)$.
\end{Theorem}

\begin{proof}
One inclusion is Proposition  \ref{prop-AComm-contained-kernel}. Conversely, take $c \in \ker(d^F)$, and let $\varphi: A \to B$ be a $F$-subnormal representative of $c$. That $c$ belongs to the kernel of $d^F$ means that $n^F_S(A) = n^F_S(B)$ for every $S \in \simple$. Therefore by Theorem \ref{thm-caract-conj-AComm} we infer that $A,B$ are conjugate in $\AComm(F)$. Since we also have $cAc^{-1} = B$, we deduce that there is $c' \in \AComm(F)$ such that $c' c  \in \Aut(A)$. Since $\Aut(A) \leq \AComm(F)$ we deduce that $c$ belongs to $\AComm(F)$, as desired. 
\end{proof}

\begin{Corollary} \label{cor-ker-CommF-monolith}
	The monolith of the group $\Comm(F)$ is equal to the kernel of $d^F : \Comm(F) \to \Z^{(\simple)}$. In particular every proper quotient of $\Comm(F)$ is abelian.
\end{Corollary}

\begin{proof}
Follows from Theorem \ref{thm-ker-CommF} and Theorem 5.4 in \cite{BELBRVW-comm-free}, which asserts that $\AComm(F)$ is the monolith of the group $\Comm(F)$. 
\end{proof}

\subsection{On the abstract commensurator of surface groups} \label{subsec-comm-surface}

We end this section with a discussion that compares the situation where the free group $F$ is replaced by the fundamental group $\pi_1(\Sigma)$ of a closed surface $\Sigma$ of genus $\geq 2$. Since any two groups $\pi_1(\Sigma)$ are virtually isomorphic, the group $\Comm(\pi_1(\Sigma))$ does not depend on the genus. Like in the free group case, the group $\Comm(\pi_1(\Sigma))$ is known to be monolithic, with simple monolith \cite[Theorem A.2]{Cap-appendix-comm-tree}.

The arguments in the free group case in the proof of Proposition \ref{prop-image-d^F} apply equally to $d^{\pi_1(\Sigma)} : \Comm(\pi_1(\Sigma)) \to \Z^{(\simple)}$. Hence we have: 
%

\begin{Proposition} \label{prop-image-d-surface}
	Let $\pi_1(\Sigma)$ be the fundamental group of a closed surface of genus at least $2$. Then the image of $d^{\pi_1(\Sigma)} : \Comm(\pi_1(\Sigma)) \to \Z^{(\simple)}$ is \[ d^{\pi_1(\Sigma)} (\Comm(\pi_1(\Sigma))) = \left\lbrace (x_S)_\simple \in \Z^{(\simple)} \, : \, \prod_\simple  |S|^{x_S}= 1 \right\rbrace. \] In particular it is a free abelian group of infinite rank.
\end{Proposition}


One can ask how much other results for $\Comm(F)$ hold for $\Comm(\pi_1(\Sigma))$. For instance we do not know whether the analogue of Corollary  \ref{cor-ker-CommF-monolith} is true, i.e.\ whether every proper quotient of $\Comm(\pi_1(\Sigma))$ is abelian. There seem to be two points that would need to be addressed to have a clean correspondence between the free group and the surface group situations. The first one is that while the monolith of $\Comm(F)$ is equal to $\AComm(F)$ \cite{BELBRVW-comm-free}, we do not know whether the same is true for $\Comm(\pi_1(\Sigma))$. The monolith is contained in $\AComm(\pi_1(\Sigma))$ as $\AComm(\pi_1(\Sigma))$ is a normal subgroup of $\Comm(\pi_1(\Sigma))$ \cite[Proposition 1.4.2]{Odden-thesis},  \cite[Lemma 2.22]{BELBRVW-comm-free}, but we do not know whether there is equality. The second one, more closely connected to the present work, is that we do not know whether the analogue of Theorem \ref{thm-ker-CommF} is true here, i.e.\ whether the kernel of $d^{\pi_1(\Sigma)} : \Comm(\pi_1(\Sigma)) \to \Z^{(\simple)}$ is the subgroup $\AComm(\pi_1(\Sigma))$. Imitating the above arguments does not seem to work immediately. Indeed, in that situation it is known that given a finite simple group $S$, there are obstructions to the transitivity of the $\Aut(\pi_1(\Sigma))$-action on the set of kernels of epimorphisms $\pi_1(\Sigma) \to S$. More precisely, to every  kernel of $\pi_1(\Sigma) \to S$ there is an associated class in $H_2(S) / \mathrm{Out}(S)$ that is $\Aut(\pi_1(\Sigma))$-invariant  (see for instance the discussion in \cite[\S 10]{Lubotzky-survey-AutF-representations}). Moreover, the $\Aut(\pi_1(\Sigma))$-action on the set of kernels of epimorphisms $\pi_1(\Sigma) \to S$ corresponding to a given $[c] \in H_2(S) / \mathrm{Out}(S)$ can also fail to be transitive \cite{Funar-Lochak}. On the positive side, let us mention however that the analogue of Proposition  \ref{prop-transitive-AutF-large-rank} for $S$ \textit{and} $[c] \in H_2(S) / \mathrm{Out}(S)$ fixed holds true here: the $\Aut(\pi_1(\Sigma))$-action on the set of kernels of epimorphisms $\pi_1(\Sigma) \to S$ corresponding to $[c]$ is transitive for large enough genus \cite{Dunfield-Thurston-random-3man}.

\section{The commensurator of a cocompact tree lattice} \label{sec-comm-tree}

\subsection{Preliminaries} \label{subsec-prelim-W-Cd}

Let $d \geq 3$, and $T_d$ the $d$-regular tree. Let $W$ be the free product of $d$ copies of the cyclic group of order $2$ (the free Coxeter group of rank $d$). We fix $a_1, \ldots, a_d \in W$ such that $a_i^2 = 1$ and $a_1, \ldots, a_d$ generate $W$. We systematically view $W$ as being equipped with this finite generating subset. We identify the tree $T_d$ with the Cayley graph of $W$.  Let $v_0$ be the vertex of $T_d$ corresponding to the identity element. There is a coloring of the edges of $T_d$ by the elements of $\left\lbrace 1, \ldots,d \right\rbrace $: an edge between two elements $\gamma$ and $\gamma a_i$ of $W$ has color $i$. The left action of $W$ on itself allows to see $W$ as a subgroup of $\Aut(T_d)$. It is a  cocompact lattice because $W$ acts freely and transitively on vertices of $T_d$.

\begin{Definition}
For $g \in \Aut(T_d)$ and $v$ a vertex of $T_d$, we denote by $\sigma(g,v)$ the permutation of $\left\lbrace 1, \ldots,d \right\rbrace $ defined by $g(va_i) = g(v) a_{\sigma(g,v)(i)}$. 
\end{Definition}

It is easily seen that every element $g \in \Aut(T_d)$ is characterized by the data of $g(v_0)$ and the collection of permutations $\left\lbrace \sigma(g,v) \right\rbrace$ where $v$ ranges over all vertices of $T_d$. Also we note that $W$ is precisely the subgroup of $\Aut(T_d)$ consisting of elements $g$ such that $ \sigma(g,v) = 1$ for every vertex $v$. 

\begin{Definition}
	We denote by $C_d$ the commensurator of $W$ in $\Aut(T_d)$.
\end{Definition}

 Any cocompact lattice  $\Gamma$ of $\Aut(T_d)$ admits a conjugate that is commensurable with $W$ in $\Aut(T_d)$, and hence has a commensurator in $\Aut(T_d)$ that is conjugate to $C_d$ \cite[4.15-4.17]{Bass-Kulk-90}. Hence it will be enough to prove Theorem  \ref{thm-intro-comme-tree-abelian-quotient} from the introduction for $\Gamma = W$. 

The following is Proposition 2.2 in \cite{LMZ-superrigidity}. 

\begin{Lemma} \label{lem-description-periodic}
Let $g \in \Aut(T_d)$ and $\gamma \in W$. Then $\gamma \in W \cap g^{-1} W g$ if and only if $\sigma(g,v) = \sigma(g,\gamma v)$ for every vertex $v$. 
\end{Lemma}

We will make use of the interpretation of elements of the stabilizer of $v_0$ in $C_d$ as recoloring of finite graphs from Section 2 in \cite{LMZ-superrigidity}, notably  Propositions 2.2--2.3-2.-4 there. However we will be working in a slightly different setting than \cite{LMZ-superrigidity}, as we want to allow an edge to be a loop in the finite graphs we consider. This requires slight modifications of the framework from \cite{LMZ-superrigidity}. For completeness we explain the details. 

Let $X$ be a finite connected $d$-regular graph (every vertex has $d$ adjacent edges). Edges are non-oriented. We allow an edge to be a loop, and we also allow multiple edges. We fix $x_0$ a base  vertex of $X$. By a coloring $c$ of $X$ we mean an assignment for each edge of $X$ of an element in $\left\lbrace 1, \ldots,d \right\rbrace $ (the color of that edge) such that for every vertex, the $d$ adjacent edges have different colors. If $c$ is a coloring of $X$, there exists a unique color preserving map $\pi: T_d \to (X,c)$ such that $\pi(v_0) = x_0$, and this map is surjective. Since we allow an edge to be a loop in $X$, the map $\pi$ is not necessarily a covering.  

Associated to a coloring $c$ of $X$ there is an action of $W$ on the vertex set $VX$. It is defined by declaring that $a_i$ flips all pairs of distinct vertices joined by an edge with color $i$, and $a_i$ fixes all vertices for which the adjacent edge with color $i$ is a loop. By the universal property of free products, this indeed defines an action of $W$ on $VX$. The graph $X$ with the given coloring is then the Schreier graph associated to the $W$-action on $VX$. The stabilizer of the base vertex $x_0$ in $W$ will be denoted $W_{x_0,c}$. Note that conversely, for every action of $W$ on a finite set, the associated Schreier graph is a finite connected $d$-regular graph equipped with a coloring. 

\begin{Lemma} \label{lem-description-fibers}
Let $X$ be a finite connected $d$-regular graph, and $x_0$ a base vertex of $X$. Let $c$ be a coloring of $X$, and $\pi: T_d \to (X,c)$ the unique color preserving map such that $\pi(v_0) = x_0$. Let $x$ be a vertex of $X$. If $(j_1, \ldots, j_r)$ is the labelling of a path in $(X,c)$ going from $x_0$ to $x$, then the fiber of $x$ under $\pi$ is \[ \pi^{-1}(x) = W_{x_0,c} \, a_{j_1} \cdots a_{j_r}, \] where $W_{x_0,c}$ is the stabilizer of $x_0$ in $W$ for the action of $W$ on $VX$ defined by $c$.
\end{Lemma}

\begin{proof}
	Take $\gamma \in W$ with normal form $\gamma = a_{i_1} \cdots a_{i_k}$. Then $\gamma \in \pi^{-1}(x)$ if and only if the path from $x_0$ in $(X,c)$ labelled $(i_1, \cdots, i_k)$ terminates at $x$. Since $(j_1, \ldots, j_r)$ is the labelling of a path in $(X,c)$ going from $x_0$ to $x$, this conditions is also equivalent to the fact that the path from $x_0$ in $(X,c)$ labelled $(i_1, \cdots, i_k, j_r, \ldots, j_1)$ is a loop. This last condition means that $w := a_{j_1} \cdots a_{j_r} a_{i_k} \cdots a_{i_1}$ belongs to $W_{x_0,c}$. Since $w= a_{j_1} \cdots a_{j_r} \gamma^{-1}$, the condition $w  \in W_{x_0,c}$ is equivalent to $\gamma \in  W_{x_0,c} \,  a_{j_1} \cdots a_{j_r}$. 
\end{proof}

Let now $(c_1,c_2)$ be two colorings of $X$. The pair $(c_1,c_2)$ gives rise to a collection of permutations $(\sigma_x)_{x \in VX}$ of $\left\lbrace 1, \ldots,d \right\rbrace $, defined by the property that for every $x \in VX$ and every edge $e$ adjacent to $x$, the coloring $c_2$ is given by $c_2(e) = \sigma_x(c_1(e))$.

The following proposition is the analogue of Proposition 2.4 in \cite{LMZ-superrigidity} in our present setting. 

\begin{Proposition} \label{prop-commens-tree-recoloring}
	Let $X$ be a finite connected $d$-regular graph, and $x_0$ a base vertex of $X$. Let $c_1$ be a coloring of $X$,  and $\pi: T_d \to (X,c_1)$ the unique color preserving map such that $\pi(v_0) = x_0$. Let $c_2$ be another coloring of $X$, and let $(\sigma_x)_{x \in VX}$ be the collection of permutations of $\left\lbrace 1, \ldots,d \right\rbrace $ associated to the pair $(c_1,c_2)$. Then the following hold: \begin{enumerate}
		\item \label{item-construction-elements}  The element $g$ of $\Aut(T_d)$ defined by the conditions $g(v_0) = v_0$ and $\sigma(g,v) = \sigma_{\pi(v)}$ for every vertex $v$ of $T_d$, belongs to $C_d$. 
		\item \label{item-path-conjugate} For every loop in $X$ based at $x_0$ having labelling $(i_1, \ldots, i_k)$ for $c_1$ and $(j_1, \ldots, j_k)$ for $c_2$, one has $g a_{i_1} \cdots a_{i_k} g^{-1} = a_{j_1} \cdots a_{j_k}$. 
		\item \label{item-stab-conjugate}  If $W_{x_0,c_i}$ is the stabilizer of $x_0$ in $W$ for the action of $W$ on $VX$ associated to $c_i$, then $gW_{x_0,c_1}g^{-1} = W_{x_0,c_2}$. 
	\end{enumerate}
\end{Proposition}

\begin{proof}
(\ref{item-construction-elements}). By definition of the element $g$, the permutations $\sigma(g,v)$ are constant along fibers of $\pi$. According to Lemma \ref{lem-description-fibers}, each fiber of $\pi$ is a right coset of $W_{x_0,c_1}$. Hence for the left action of $W$ on itself, the subgroup $W_{x_0,c_1}$ preserves each fiber of $\pi$. We deduce that for every $\gamma \in W_{x_0,c_1}$, one has $\sigma(g,v) = \sigma(g, \gamma v)$ for every vertex $v$. According to Lemma \ref{lem-description-periodic} this implies $W_{x_0,c_1} \leq W \cap g^{-1} W g$. Since the graph $X$ is finite, $W_{x_0,c_1}$ is a finite index subgroup of $W$. We deduce $W \cap g^{-1} W g$ has finite index in $W$. Therefore $W \cap g^{-1} W g$ is a cocompact lattice in  $\Aut(T_d)$. Conjugation by $g$ sends $W \cap g^{-1} W g$ to $gW g^{-1} \cap W$, so $gW g^{-1} \cap W$ is also a cocompact lattice. This forces $gW g^{-1} \cap W$ to be of finite index in $W$ as well. Therefore $g \in C_d$. 

(\ref{item-path-conjugate}). Set $\gamma = a_{i_1} \cdots a_{i_k}$. Since $(i_1, \ldots, i_k)$ is the labelling for $c_1$ of a loop at $x_0$, so is $(i_k, \ldots, i_1)$. Hence $\gamma \in W_{x_0,c_1}$. At the level of the tree, the element $\gamma$ corresponds to the path from $v_0$ labelled  $(i_1, \ldots, i_k)$. According to the previous paragraph one has $g W_{x_0,c_1} g^{-1} \leq W$. Therefore $g \gamma g^{-1} \in W$. Since $g$ fixes $v_0$, the element $g \gamma g^{-1}$ sends $v_0$ to the vertex $g(\gamma)$. Hence the normal form of $g \gamma g^{-1}$ is obtained by following the path in $T_d$ going from $v_0$ to the vertex $g(\gamma)$. By construction of the element $g$, this path is labelled $(j_1, \ldots, j_k)$. Hence $g \gamma g^{-1} = a_{j_1} \cdots a_{j_k}$. 

(\ref{item-stab-conjugate}) follows from (\ref{item-path-conjugate}). 
\end{proof}

\subsection{Proof of Theorem \ref{thm-intro-comme-tree-abelian-quotient}} \label{subsec-proof-comm-tree}

\begin{proof}[Proof of Theorem \ref{thm-intro-comme-tree-abelian-quotient}]
	As explained above, it is enough to prove the theorem for the group $C_d$. The only element of $\Aut(T_d)$ centralizing a finite index subgroup of $W$ is the identity, so the natural homomorphism $C_d \to \Comm(W)$ from $C_d$ to the abstract commensurator of $W$ is injective \cite[B.7]{Bass-Kulk-90}. In the sequel we identify $C_d$ with its image in $\Comm(W)$. We consider the homomorphism $d^W : \Comm(W) \to \Z^{(\simple)}$ from Theorem \ref{thm-definition-morphism}. We will show that the image of the restriction of $d^W$ to $C_d$  has infinite rank. This will prove the statement. 	We fix an integer $r \geq 5$. We denote by $A(r)$ the alternating group $A(r) = \mathrm{Alt}(r)$. Using notation from Definition \ref{defi-map-dS}, we will show that $d_{A(r)}^W: C_d \to \Z$ is surjective. This will imply the image of  $d^W : C_d \to \Z^{(\simple)}$ is infinitely generated, and hence is free abelian of infinite rank. 
	
	Let $X$ be the $d$-regular graph with $r$ vertices $x_0, \ldots, x_{r-1}$, with an edge $e_i$ between $x_i$ and $x_{i+1}$ for every $i=0, \dots, r-2$, and such that all other edges are loops. We consider a coloring $c_1$ of $X$ such that edges $e_0, e_2, \ldots, e_{r-2}$ have color $1$, and edges  $e_1, e_3, \ldots, e_{r-3}$ have color $2$. How $c_1$ is defined on other edges will not matter. The permutation group on $VX$ induced by the $W$-action associated to $c_1$ is the dihedral group of order $2r$. 
	
	We now define a coloring $c_2$ of $X$ by starting from $c_1$, and specifying for each vertex $x$ of $X$ a permutation of $\left\lbrace 1, \ldots,d \right\rbrace $. For vertices $x_3, \ldots, x_{r-1}$, we take the identity. In other words $c_2$ coincide with $c_1$ on all edges adjacent to those vertices. For $x_0, x_1, x_{2}$, we take respectively $(12), (123), (23)$. This indeed defines  a coloring of $X$ in a coherent way as the edge joining $x_2$ and $x_3$ has color $1$ for $c_1$.  Exactly five edges have changed color between $c_1$ and $c_2$, among which $e_0,e_1$, which were previously colored $1$ and $2$ and are now colored $2$ and $3$. The element $a_3$ acts on $VX$ as the transposition $(x_1,x_2)$. The supports of $a_1$ and $a_3$ intersect along $\left\lbrace x_2 \right\rbrace $, and the supports of  $a_2$ and $a_3$ intersect along $\left\lbrace x_1 \right\rbrace $. One then  easily verifies that the permutation group on $VX$ induced by the $W$-action associated to $c_2$ is the entire symmetric group $\mathrm{Sym}(r)$. 
		
	We denote by  $A =  W_{x_0,c_1}$ the stabilizer of $x_0$ in $W$ associated to the $W$-action on $VX$ defined by $c_1$, and by  $B =  W_{x_0,c_2}$ the stabilizer of $x_0$ in $W$ associated to the $W$-action on $VX$ defined by $c_2$. Proposition \ref{prop-commens-tree-recoloring} applied to $X,x_0,c_1,c_2$  provides $g \in C_d$ such that $g A g^{-1} = B$. We claim that  $d_{A(r)}^W(g)$ is equal to $1$. For, we need to find a $W$-subnormal representative of $g$. Let $A_1 = \mathrm{Core}_W(A)$ and $B_1 = \mathrm{Core}_W(B)$ denote the normal core of $A$ and $B$ respectively in $W$. By the above discussion we have $(A:A_1) = 2$ and $(B:B_1) = (r-1) !$. Let $A_2 = g^{-1} B_1 g \cap A_1$ and $B_2 = g A_2 g^{-1} = B_1 \cap g A_1 g^{-1}$. Since $B_1 \lhd B$ we have $g^{-1} B_1 g \lhd A$, and hence $A_2 \lhd A_1$. Therefore $A_2 \lhd A_1 \lhd W$. Similarly since $A_1 \lhd A$ we have $g A_1 g^{-1} \lhd B$, and hence $B_2 \lhd B_1$. Therefore $B_2 \lhd B_1 \lhd W$. So all together the restriction of the conjugation by $g$ to $A_2$ defines an isomorphism $A_2 \to B_2$ that is a $W$-subnormal representative of $g$. We compute $(n_S^W(B_2))_\simple$. The group $W/B_1$ is the symmetric group on $r$ elements and we assume $r \geq 5$, so $n_S^W(B_1) = 1$ if $S \in \left\lbrace C_2,  A(r) \right\rbrace$, and  $n_S^W(B_1) = 0$ otherwise. Since $(A:A_1) = 2$ we have $(B:g A_1 g^{-1}) = 2$, and hence $(B_1:B_2) \leq 2$. So $B_1/B_2$ is either trivial or cyclic of order two. By Lemma \ref{lem-chasles} applied to $B_2 \lhd B_1 \lhd W$ we infer  $n_S^W(B_2) = 1$ if $S= A(r)$,  $n_S^W(B_2) \in \left\lbrace 1,2\right\rbrace $ if $S= C_2$, and  $n_S^W(B_2) = 0$ otherwise. In order to complete the proof one needs to see $n_S^W(A_2) = 0$ if $S= A(r)$. Since $W/A_1$ is a solvable dihedral group, $n_S^W(A_1) = 0$ for $S$ non-abelian. Now one has $(A: A_2) = (A: A_1)  (A_1: A_2)  = 2 (A_1: A_2)$. On the other hand $(A: A_2) = (B:B_2)$, which is equal to either $(r-1) !$ or $2 (r-1) !$ (depending on whether $B_1/B_2$ is trivial or of order two). Combined together we deduce $(A_1: A_2) \leq (r-1) !$. In particular Lemma \ref{lem-formula-index} prevents $A(r)$ from appearing in a composition series going from $A_2$ to $A_1$. In other words $n_S^{A_1}(A_2) = 0$ if $S= A(r)$. By Lemma \ref{lem-chasles} applied to $A_2 \lhd A_1 \lhd W$ we infer  $n_S^W(A_2) = 0$ if $S= A(r)$, as desired. This completes the proof.
\end{proof}

\subsection{Further questions} 

Recall from the introduction that the group $C_d$ is known to be monolithic, and its monolith $M_d := \mathrm{Mon}(C_d)$ is a simple group \cite{Cap-appendix-comm-tree}. The infinite index normal subgroup $N_d$ of $C_d$ that we exhibit in the proof of Theorem \ref{thm-intro-comme-tree-abelian-quotient} is the kernel of the  homomorphism $C_d \to \Z^{(\simple)}$ obtained as the composition of the  embedding of $C_d$ in the abstract commensurator $\Comm(W)$ and the homomorphism $d^W : \Comm(W) \to \Z^{(\simple)}$. By definition of the monolith we have $M_d \leq N_d$. While for the abstract commensurator $\Comm(W)$ the monolith is equal to the kernel of $d^W : \Comm(W) \to \Z^{(\simple)}$ by Corollary \ref{cor-ker-CommF-monolith} (since $W$ is virtually a free group and by Lemma \ref{prop-compatibility-hom-finite-index}), we do not know whether the same is true for $C_d$. In other words, we do not know if $M_d = N_d$. More generally it would be interesting to have a description of $C_d / M_d$ and of $M_d$. In that direction, we ask the following questions: 

\begin{Question}
	Is the quotient $C_d / M_d$ is virtually abelian ? 
\end{Question}

Denote by $C_d^+$  the index two subgroup of $C_d$ that acts on $T_d$ in a type preserving way. Recall that $M_d$ contains $W \cap C_d^+$ \cite[Proposition 5.1]{LMZ-superrigidity}. A positive answer to the following question could be seen as an analogue of Theorem 5.4 in \cite{BELBRVW-comm-free} for the group $C_d$. 

\begin{Question}
Is the monolith $M_d$ equal to the subgroup of $C_d^+$ generated by subgroups of $C_d^+$ that are commensurable with $W \cap C_d^+$ ? 
\end{Question}

The subgroup described above is normal in $C_d$, and hence contains $M_d$. The question asks if it is contained in $M_d$. 

\subsection{Consequences of Theorem \ref{thm-intro-comme-tree-abelian-quotient}} 

Let $\Gamma$ be a cocompact lattice in $G = \Aut(T_d)$, and $C = \Comm_G(\Gamma)$ its commensurator. 
Let $ C / \! \! / \Gamma$ denote the completion of $C$ with respect to $\Gamma$, and  $\tau_{C,\Gamma}: C  \to C / \! \! / \Gamma$ the associated homomorphism. We refer to Section 3 in \cite{Shalom-Willis} for basic properties. The group  $C / \! \! / \Gamma$ is a totally disconnected locally compact group. The image of $C$ in $ C / \! \! / \Gamma$ is dense, and the subgroup $U :=  \overline{ \tau_{C,\Gamma}(\Gamma)}$ is a compact open subgroup of $ C / \! \! / \Gamma$ with the property that $\tau_{C,\Gamma}^{-1}(U) = \Gamma$.

\begin{Corollary} \label{cor-completion-abelian-quotient}
	Let $\Gamma$ be a cocompact lattice in $G = \Aut(T_d)$, and $C = \Comm_G(\Gamma)$. Then the group $ C / \! \! / \Gamma$  admits an open normal subgroup such that the associated quotient is a free abelian group of infinite rank. In particular $ C / \! \! / \Gamma $  is not virtually topologically simple.
\end{Corollary}

\begin{proof}
Let $N$ be the normal subgroup of $C$ afforded by Theorem \ref{thm-intro-comme-tree-abelian-quotient}. So $C/N$ is free abelian of infinite rank. Moreover there is a finite index subgroup of $\Gamma$ that is contained in $N$. This follows by the construction of $N$, but this is also consequence of \cite[Proposition 5.1]{LMZ-superrigidity}, or \cite[Theorem 1.1]{Creutz-Shalom-NST}, or \cite[Theorem A.1]{Cap-appendix-comm-tree}. Since $C/N$ is torsion free, actually $\Gamma \leq N$. It follows that the $C$-equivariant map $C / \Gamma \to C/N$ induces a continuous surjective homomorphism $C / \! \! / \Gamma \to C/N$ with kernel  $\overline{ \tau_{C,\Gamma}(N)}$ \cite[Lemma 3.8]{Shalom-Willis}. This shows the statement. 
\end{proof}

\bibliographystyle{amsalpha}
\bibliography{general-bibliography}
\end{document}